\documentclass[10pt]{article}
\usepackage[utf8]{inputenc}
\usepackage[dvips]{graphicx}
\usepackage{epsfig,color}
\usepackage[USenglish]{babel}
\usepackage{amsmath}
\usepackage{amsthm}
\usepackage{amssymb}
\usepackage{tikz}
\usepackage{hyperref}
\usepackage[parfill]{parskip}

\makeatletter
\def\thm@space@setup{%
  \thm@preskip=\parskip \thm@postskip=0pt
}
\makeatother

\newtheorem{defn}{Definition}[section]
\newtheorem{teo}[defn]{Theorem}
\newtheorem{lema}[defn]{Lemma}

\newtheorem{conj}{Conjecture}
\newtheorem*{conjj}{Lescure and Meyniel's Conjecture}

\theoremstyle{definition}
\newtheorem*{obs}{Observation}

\newcommand{\imm}{\preccurlyeq_i}

\begin{document}

\title{Complete graph immersions in dense graphs}
\author{Sylvia Vergara S. \\ Universidad de Chile}
\date{\today}
\maketitle

\begin{abstract}
In this article we consider the relationship between vertex coloring and the immersion order. Specifically, a conjecture proposed by Abu-Khzam and Langston in 2003, which says that the complete graph with $t$ vertices can be immersed in any $t$-chromatic graph, is studied.

First, we present a general result about immersions and prove that the conjecture holds for graphs whose complement does not contain any induced cycle of length four and also for graphs having the property that every set of five vertices induces a subgraph with at least six edges. 

Then, we study the class of all graphs with independence number less than three, which are graphs of interest for Hadwiger's Conjecture. We study such graphs for the immersion-analog. If Abu-Khzam and Langston's conjecture is true for this class of graphs, then an easy argument shows that every graph of independence number less than $3$ contains $K_{\left\lceil\frac{n}{2}\right\rceil}$ as an immersion. We show that the converse is also true. That is, if every graph with independence number less than $3$ contains an immersion of $K_{\left\lceil\frac{n}{2}\right\rceil}$, then Abu-Khzam and Langston's conjecture is true for this class of graphs. 
Furthermore, we show that every graph of independence number less than $3$ has an immersion of $K_{\left\lceil\frac{n}{3}\right\rceil}$.
\end{abstract}

\section{Introduction}

Vertex coloring has been a very important topic in graph theory. The usual goal, and the one considered here, is to color every vertex of a graph such that adjacent vertices get different colors. The \textit{chromatic number} of a graph $G$, denoted $\chi(G)$, is the minimum number of colors required to color its vertices. If $\chi(G)=t$, then we say that $G$ is \textit{$t$-chromatic}. 

It has been suspected for a long time that if a graph cannot be colored with $t-1$ colors, then it has to somehow contain the complete graph $K_t$ with $t$ vertices. At some point in the 40's, Haj\'os \cite{HajosConjecture} conjectured that the relation of containment was the topological order. This conjecture is true for $t \leq 4$ \cite{DiracHajosConjecture}, but false for $t \geq 7$ \cite{CatlinHajosConjecture}. It remains open for $t \in \{5,6\}$. In 1943 Hadwiger \cite{HadwigerConjecture} suggested that the containment had to be the minor order, i.e. he conjectured that every $t$-chromatic graph contains $K_t$ as a minor. It was shown that Hadwiger's conjecture holds for $t=5$ \cite{WagnerHadwigerConjecture} and $t=6$ \cite{RSTHadwigerConjecture}. But it remains uncertain whether or not the conjecture is true for $t \geq 7$. 

In this article we study a different order, the immersion order, which is defined by lifts of edges. A \textit{lift} of two (adjacent) edges $uv$ and $vw$, with $u \neq w$ and $uw \notin E(G)$, consists of deleting $uv$ and $vw$, and adding the edge $uw$. And a graph $H$ is \textit{immersed} in a graph $G$ if $H$ can be obtained from $G$ by performing lifts of edges and deleting vertices and/or edges. We denote this by $H \imm G$. We also say that $G$ contains an \textit{immersion} of $H$. This definition is equivalent \cite{GraphColImm} to the existence of an injective function $\phi:V(H)\rightarrow V(G)$ such that:

\begin{enumerate}
\item For every $uv \in E(H)$, there is a path in $G$, denoted $P_{uv}$, which connects $\phi(u)$ and $\phi(v)$.

\item The paths $\{P_{uv} : uv \in E(H)\}$ are pairwise edge-disjoint.
\end{enumerate}

If the paths $P_{uv}$ are internally disjoint from $\phi(V(H))$, then we say that the immersion is \textit{strong}. We call the vertices in $\phi(V(H))$ the \textit{corner vertices} of the immersion. 

Clearly topological containment implies immersion containment (strong immersion containment, actually). However, the minor order and the immersion order are not comparable. The immersion order, although initially much less studied than the minor and topological orders, has received a large amount of attention recently \cite{Booth1999344, Fellows1988727, Fellows:1992:WTA:131829.131839, Fellows1994769, ForbKuratImm, Langston1998191, StructureNotFixedImm}. In fact, Robertson and Seymour extended their proof of Wagner's famous conjecture \cite{GraphMinorTheorem}, to prove that the immersion order is a well-quasi-order \cite{GraphMinorTheoremImmersion}. 

In analogy to Hadwiger and Haj\'os' conjectures, Lescure and Meyniel \cite{Lescure1988325} conjectured the following.

\begin{conjj}
If $\chi(G) \geq t$, then $G$ contains a strong immersion of $K_t$. 
\end{conjj}

Independently, Abu-Khzam and Langston \cite{GraphColImm} proposed a weaker statement.

\begin{conj}[Abu-Khzam and Langston] \label{conj1}
If $\chi(G) \geq t$, then $K_t$ is immersed in $G$.
\end{conj} 

Since Haj\'os' conjecture holds for $t \leq 4$, Abu-Khzam and Langston's conjecture is true for $t \leq 4$, as topological order is just a particular case of immersion order.

Each graph $G$ with $\chi(G)=t$ must contain a \textit{$t$-critical} subgraph, i.e., a graph $\widetilde{G}$ such that $\chi(\widetilde{G})=t$ and $\chi(H)<t$ for every proper subgraph $H$ of $\widetilde{G}$. Furthermore, it is easy to see that every $t$-critical graph must have minimum degree at least $t-1$. Using this fact, DeVos, Kawarabayashi, Mohar and Okamura \cite{ImmSmall} resolved Abu-Khzam and Langston's conjecture for small values of $t$.

\begin{teo}[\cite{ImmSmall}] \label{ALchico}
Let $f(k)$ be the smallest integer such that every graph of minimum degree at least $f(k)$ contains an immersion of $K_k$. Then $f(k)=k-1$ for $k \in \{5,6,7\}$.
\end{teo}

For $k \geq 8$, however, $f(k) \geq k$ \cite{NoImm, MinDegImm}, i.e. $\delta(G) \geq k-1$ does not guarantee an immersion of $K_k$ in $G$.

Theorem~\ref{ALchico} solves Abu-Khzam and Langston's conjecture for very small values of $t$. We are interested here in the other end of the spectrum, where $t$ is close to the number of vertices. So we restrict our attention to classes of graphs which are quite dense. We already know some properties about dense graphs, such that if a graph has $2cn^2$ edges, then it contains a strong immersion of the complete graph on at least $c^2n$ vertices \cite{MinDegImm}.

One very special case of dense graphs are the complete multipartite graphs. We prove the following result.

\begin{teo} \label{multi}
Let $G$ be a complete multipartite graph of $k\geq2$ classes with $s$ vertices each. Then $G$ has a strong immersion of $H$, where,
$$H =
\begin{cases}
K_{(k-1)s+1} & \text{ if } s \text{ is even}\\ 
K_{(k-1)s} & \text{ if } s\neq1 \text{ and } s \text{ is odd} \\ 
K_k & \text{ if } s=1 
\end{cases}	
$$
\end{teo}

We will call a graph \textit{$(k,s)$-dense} if every set of $k$ vertices induces a subgraph with at least $s$ edges. We prove the following two results. 

\begin{teo} \label{Teo56Denso}
Every (5,6)-dense graph $G$ contains a strong immersion of $K_{\chi(G)}$.
\end{teo}

\begin{teo} \label{TeoC4Inducido}
Any graph $G$ whose complement has no induced cycle of length four contains a strong immersion of $K_{\chi(G)}$.
\end{teo}

Finally, we focus on the study of a special class of graphs, the graphs $G$ that have no independent set of size three, or equivalently, whose independence number $\alpha(G)$ is at most $2$. This class of graphs has been extensively studied in an attempt to solve Hadwiger's conjecture (see \cite{ASpecialCaseHadwiger, HadwigerSeagullPacking, PackingSeagulls, OnASpecialCaseHadwiger}). It is for this reason that we are interested in Abu-Khzam and Langston's conjecture restricted to these graphs. Abu-Khzam and Langston's conjecture restricted to that class reads as follows.

\begin{conj} \label{conj2}
Any graph $G$ with $\alpha(G) \leq 2$ contains an immersion of $K_{\chi(G)}$.
\end{conj}

If $\alpha(G) \leq 2$, then in any vertex coloring of $G$, every color class, being an independent set, has at most two vertices, which implies that $\chi(G) \geq \frac{n}{2}$. Abu-Khzam and Langston's conjecture would thus imply that $G$ must contain an immersion of $K_{\lceil \frac{n}{2} \rceil}$. The latter gives rise to a new conjecture.

\begin{conj} \label{conj3}
Any graph $G$ with $\alpha(G) \leq 2$ contains an immersion of $K_{\lceil \frac{n}{2} \rceil}$.
\end{conj}

We just saw that Conjecture~\ref{conj2} implies Conjecture~\ref{conj3}. However, the two conjectures are actually equivalent. Following ideas from \cite{OnASpecialCaseHadwiger} we show the next result.

\begin{teo} \label{TeoConjeturasEquivalentes}
Conjectures~\ref{conj2} and~\ref{conj3} are equivalent.
\end{teo}  

A weaker version of Conjecture~\ref{conj3} is shown, namely the following result.

\begin{teo} \label{TeoEneTercios}
If $G$ is a graph with $\alpha(G) \leq 2$, then $G$ contains a strong immersion of $K_{\lceil\frac{n}{3}\rceil}$.
\end{teo}

An analogous result was shown by Chudnovsky \cite{HadwigerSeagullPacking}, namely that if $G$ is a graph with $n$ vertices and no independent set of size three, then $G$ contains a complete minor of size $\lceil\frac{n}{3}\rceil$. The technique used there is a nice use of induced paths of length two. Here we present a different technique.

In this article every graph is simple, without loops and parallel edges, unless stated otherwise.

This work is organized as follows. In Section~\ref{seccionDefiniciones} we present a quick review of some definitions and properties about vertex coloring that will be used through the text. In Section~\ref{seccionResultadosGenerales} we immerse a large complete graph into a multipartite complete graph (see Theorem~\ref{multi}), and also prove Theorems~\ref{Teo56Denso} and~\ref{TeoC4Inducido}. And in Section~\ref{SeccionAlpha} we prove Theorems~\ref{TeoConjeturasEquivalentes} and~\ref{TeoEneTercios}, and show a series of properties that a counterexample of Conjecture~\ref{conj2} with minimum number of vertices should satisfy.

\section{Vertex coloring} \label{seccionDefiniciones}

Given a vertex coloring $c:V(G) \rightarrow \{1,...,k\}$, we denote $c_i = \{u: c(u)=i\}$ and $c_{ij}$ the subgraph induced by the set of vertices $\{u: c(u) \in \{i, j\}\}$. 
We call a path in $c_{ij}$ a \textit{chain}, and for each $u \in V(c_{ij})$, we denote $c_{ij}(u)$ the connected component of $c_{ij}$ that contains $u$.
If $\{i,j\} \neq \{k,l\}$, then $c_{ij}$ and $c_{kl}$ are edge-disjoint graphs.
This observation is particularly important to find immersions in graphs, considering the second definition of immersion. For this reason, the use of chains will be very helpful.

Let $c:V(G) \rightarrow \{1,...,k\}$ be a vertex coloring of $G$ and let $i \in \{1,...,k\}$. We say that $u \in V(G)$ is a \textit{dominating vertex} for color $i$, if $c(u)=i$ and if for each color $j \neq i$, there is a vertex $v$ such that $c(v)=j$ and $uv \in E(G)$.
If $c:V(G) \rightarrow \{1,...,\chi(G)\}$ is a coloring of $G$ with minimum number of colors, then it is easy to check that every $i \in \{1,...,\chi(G)\}$ has a dominating vertex. 

\section{Complete Graph Immersions} \label{seccionResultadosGenerales}

Let us see first, that in a complete multipartite graph we can find an immersion of a complete graph of relatively large size. That is, let us prove Theorem~\ref{multi}.

\begin{proof}[Proof of Theorem~\ref{multi}]
The $s=1$ case is trivial, so we can assume $s>1$. We choose the vertices of $ k-1$ classes as corner vertices (in the case that $s$ is even, we will add an additional corner vertex later), and the vertices of the remaining class, let us call it $U$, will be used for the edge-disjoint paths. The paths between two vertices from different classes already exist (they are the edges between them), so we only need to worry about those vertices that are in the same class. We know that $\chi'(K_s)= s-1$ if $s$ is even, and $\chi'(K_s) = s$ if $s$ is odd (\cite[p.133]{Soifer}).

For each class of $s$ corner vertices, consider a $\chi'(K_s)$-edge-coloring of the edges that are missing (all of them). As $|U| \geq \chi'(K_s)$, we can assign each of the used colors on the edges of $K_s$ to some vertex in $U$. Say vertex $u_i \in U$ gets color $i$. Then, for two corner vertices $v$ and $w$ in the same class, we assign $P_{vw} = vu_iw$ where $vw$ is colored with color $i$.

Observe that these paths are edge-disjoint. Indeed, if two paths $P_{vw}$ and $P_{xy}$ share an edge, then $vw$ would have to be adjacent to $xy$. In addition, we would have $P_{vw} = vu_iw$, $P_{xy}=xu_iy$ for some $i \leq \chi'(G)$. That is, both $vw$ and $xy$ would have assigned color $i$, which is a contradiction.

Note that if $s$ is even, then in $U$ there is a vertex that is not being used in the edge-disjoint paths, so we can add it as a corner vertex of the immersion, as it is adjacent to all other corner vertices. Thus, we find the desired immersion, which is strong because no corner vertex is used as an internal vertex of some path.
\end{proof}

\begin{obs}
Actually, a more general result follows directly from the proof of the theorem. If $G$ is a complete multipartite graph with $k\geq2$ classes of sizes $s_1, s_2, \ldots, s_k$, with $s_k \geq s_i$, for $i \leq k-1$, then $G$ contains a strong immersion of $K_{s_1+s_2+\ldots+s_{k-1}}$.
\end{obs}

We now prove Theorem~\ref{Teo56Denso}.  

\begin{proof}[Proof of Theorem~\ref{Teo56Denso}]
Let us suppose first that $G$ has fewer than five vertices. The cases $\chi(G) \in \{1,2\}$ are trivial. If $\chi(G)=3$, $G$ must contain a triangle, so $K_3 \subseteq G$. And if $\chi(G)=4$, it is easy to check that the only option is $G=K_4$. So, we can assume $|V(G)|\geq5$.

Let $c$ be a coloring of $V(G)$ with minimum number of colors and let $k=\chi(G)$. Note that $|c_i| \leq 3$, for $1 \leq i \leq k$, since there cannot be independent sets of size four. This, because if there were any, then, adding any other vertex, we would have a set of five vertices inducing less than six edges.

Observe that if $c_i = \{u,x\}$ and $c_j = \{v,y\}$ are such that $c_{ij}$ is not connected, then the vertices in $c_i \cup c_j$ are adjacent to all other vertices. Indeed, if $c_{ij}$ is not connected, it has exactly two edges. Then, any other vertex must be adjacent to $u$, $v$, $x$ and $y$, because of the ($5$,$6$)-density of $G$. By symmetry, there are two cases. 
\begin{itemize}
\item $uv, xy \in E(c_{ij})$, in which case every vertex in $c_i \cup c_j$ is a dominating vertex for its color.
\item $uv, vx \in E(c_{ij})$, in which case $v$ has to be the dominating vertex for color $j$, and both $u$ and $x$ are dominating vertices for color $i$.
\end{itemize}  

We choose a dominating vertex $u_t$ of each color $t$ as the corner vertices of the immersion with the extra requirement that if $i \neq j$ and $|c_i|=|c_j|=2$ with $c_{ij}$ disconnected, then we choose a pair of adjacent dominating vertices as corner vertices. Note that this choice is possible because of the above observation. Let $i,j$ be any two colors and we will show that $u_i,u_j$ are connected by a chain.

\begin{itemize}

\item If $|c_i|=1$, then $u_iu_j \in E(G)$, as $u_j$ is dominating. The edge $u_iu_j$ is the chain we want.

\item If $|c_i|=2, |c_j|=3$, then $u_iu_j \in E(G)$, due to the $(5,6)$-density of the graph.

\item If $|c_i|=3, |c_j|=3$, then considering $c_i$ plus $u_j$ and a vertex in $c_j \setminus \{u_j\}$, it holds that the induced subgraph must necesarily be a complete bipartite graph, because of the $(5,6)$-density of $G$. Then, $u_iu_j \in E(G)$.

\item If $|c_i|=2, |c_j|=2$, there are two cases. If $c_{ij}$ is connected, we can always find a chain between $u_i$ and $u_j$. And if $c_{ij}$ is not connected, then by the choice of $u_i,u_j$, it holds that $u_iu_j \in E(G)$. 
	
\end{itemize}
By symmetry, the above are all possible cases, and so, between each pair of corner vertices there is a chain that connects them, and therefore, we have found an immersion of $K_{\chi(G)}$. None of the chains we described uses another corner vertex as an internal vertex, so the immersion is strong.
\end{proof}

Let us prove now Theorem~\ref{TeoC4Inducido}.

\begin{proof}[Proof of Theorem~\ref{TeoC4Inducido}]
Let $c$ be a vertex coloring of $G$ with minimum number of colors, and choose a dominating vertex of each color as the set of corner vertices. Consider two corner vertices, $u$ and $v$, with $c(u)=i$, $c(v)=j$ and let us see that there is a chain that joins them (so we ensure that paths will be edge-disjoint).

If $uv \in E(G)$, then the edge $uv$ is the chain we want. If $uv \notin E(G)$, there are vertices $x \in c_j, y \in c_i$, such that $ux, vy \in E(G)$, because $u$ and $v$ are dominating. Also, as $C_4$ is not an induced subgraph of $\overline{G}$, necessarily $xy \in E(G)$. Thus $uxyv$ is the chain we want. Then we have an immersion of $K_{\chi(G)}$, which is strong since the paths being chains, they do not use another corner vertex as an internal vertex.
\end{proof}

\begin{obs}
At first, the condition that there are no induced cycles of length four in the complement of the graph might seem too restrictive, however, unlike in $(5,6)$-dense graphs, color classes can be arbitrarily big. Indeed, consider the graph obtained from $K_{2,n-2}$ by adding the edge between the two vertices in the class of size two. This graph has no induced cycle of length four in the complement, but any coloring with minimum number of colors contains a class of size $n-2$.
\end{obs}

\section{Graphs with small independence number} \label{SeccionAlpha}

Here we study the class of graphs that have no independent set of size three. It is easy to check that the non-neighbourhood of any vertex of a graph $G$ with $\alpha(G) \leq 2$ induces a complete graph.

We shall now see that if we replaced $K_{\lceil \frac{n}{2} \rceil}$ with $K_{\lceil \frac{n}{3} \rceil}$ in Conjecture~\ref{conj3}, then the statement is true, as claimed in Theorem~\ref{TeoEneTercios}. Moreover, either $G$ contains $K_{\lceil \frac{n}{3} \rceil}$ as a subgraph, or any set of $\lceil \frac{n}{3} \rceil$ vertices can be a set of corner vertices. Also, the immersion is strong. 

From now on we will use the following notation: $$\overline{N}(v) = V(G) \setminus \left(N(v) \cup \{v\}\right).$$

\begin{proof}[Proof of Theorem~\ref{TeoEneTercios}]

Let us define, for a vertex $v \in V(G)$ and a set $U \subseteq V(G)$, $$N_U(v)=N(v) \cap U$$ $$\overline{N}_U(v)=\overline{N}(v) \cap U.$$

If there was a vertex $v$ with $d(v)< \lfloor\frac{2n}{3}\rfloor$, then the non-neighborhood of $v$ would have size at least $\lceil\frac{n}{3}\rceil$, and as it induces a complete graph, we would have $G$ containing $K_{\lceil\frac{n}{3}\rceil}$ as a subgraph. Therefor, we can assume $\delta(G) \geq \lfloor \frac{2n}{3} \rfloor$. 

We will find an immersion of $K_{\left\lceil\frac{n}{3}\right\rceil}$ in $G$. We partition $V(G)$ into two disjoint sets $U$ and $W$, such that $|U| = \lceil\frac{n}{3}\rceil$ and $|W| = \lfloor\frac{2n}{3}\rfloor$. The vertices from the set $U$ will be the corner vertices and we denote $P_{uv}$ the path between $u$ and $v$ in the immersion, which will be constructed as follows. We arrange the pairs $\{u,v\}$ with $u,v \in U$ arbitrarily and we assign the paths of the immersion in the following way. If $uv \in E(G)$, then $P_{uv} = uv$. If $uv \notin E(G)$, then we will choose a vertex $z \in N_W(u) \cap N_W(v)$ such that $z$ has not been used at some $P_{ux}$, with $x \in U$ or some $P_{vx}$, with $x \in U$, and we will assign $P_{uv} = uzv$. Note that given the latter condition, the paths will be edge-disjoint. Furthermore, no corner vertex is used as an internal vertex of a path, so the immersion is indeed strong. Let us see that this assignment is possible (we only need to verify this for the case $uv \notin E(G)$). 

Let $uv \not\in E(G)$. We must prove that $u$ and $v$ have enough common vertices in $W$. That is to say, we need to prove the following.

$$|N_W(u)\cap N_W(v)| \geq |\overline{N}_U(u)| + |\overline{N}_U(v)| - 1$$

The term $-1$ is there because the non-existing edge $uv$ is being counted twice.

\begin{eqnarray*}
|\overline{N}_U(u)| + |\overline{N}_U(v)| - 1 & = & |U| - |\{u\}| - |N_U(u)| + |U| - |\{v\}| - |N_U(v)| - 1\\
                     & = & \left\lceil\frac{n}{3}\right\rceil - 1 - |N_U(u)| + \left\lceil\frac{n}{3}\right\rceil - 1 - |N_U(v)| - 1\\
                     & = & 2\left\lceil\frac{n}{3}\right\rceil - 3 - \left( |N(u)| - |N_W(u)| + |N(v)| - |N_W(v)| \right) \\
                     & = & 2\left\lceil\frac{n}{3}\right\rceil - 3 + |N_W(u)| + |N_W(v)| - \left( |N(u)| + |N(v)| \right) \\
                     & \leq &  2\left\lceil\frac{n}{3}\right\rceil - 3 + |N_W(u)| + |N_W(v)| - \left( \left\lfloor\frac{2n}{3}\right\rfloor + \left\lfloor\frac{2n}{3}\right\rfloor \right) \\
                     & = & 2\left\lceil\frac{n}{3}\right\rceil - 3 + |N_W(u)\cup N_W(v)| + |N_W(u) \cap N_W(v)| - 2\left\lfloor\frac{2n}{3}\right\rfloor \\
\end{eqnarray*}

Since $uv \notin E(G)$ and $\alpha(G)\leq2$, we have that for each $w \in W, uw \in E(G)$ or $vw \in E(G)$. This implies that $N_W(u) \cup N_W(v) = W$. Then,

\begin{eqnarray*}
|\overline{N}_U(u)| + |\overline{N}_U(v)| - 1 & \leq & 2\left\lceil\frac{n}{3}\right\rceil - 3 + |W| + |N_W(u) \cap N_W(v)| - 2\left\lfloor\frac{2n}{3}\right\rfloor \\
                     & = & 2\left\lceil\frac{n}{3}\right\rceil - 3 + \left\lfloor\frac{2n}{3}\right\rfloor + |N_W(u) \cap N_W(v)| - 2\left\lfloor\frac{2n}{3}\right\rfloor \\
                     & = & 2\left\lceil\frac{n}{3}\right\rceil - 3 - \left\lfloor\frac{2n}{3}\right\rfloor + |N_W(u) \cap N_W(v)| \\
\end{eqnarray*}

So, we only need to prove that $2\left\lceil\frac{n}{3}\right\rceil - 3 - \left\lfloor\frac{2n}{3}\right\rfloor \leq 0$. 

\begin{eqnarray*}
2\left\lceil\frac{n}{3}\right\rceil - 3 - \left\lfloor\frac{2n}{3}\right\rfloor & \leq & 2\left\lfloor\frac{n}{3}+1\right\rfloor - 3 - \left\lfloor\frac{2n}{3}\right\rfloor \\
                     & = & 2\left\lfloor\frac{n}{3}\right\rfloor + 2 - 3 - \left\lfloor\frac{2n}{3}\right\rfloor \\
                     & \leq & \left\lfloor\frac{2n}{3}\right\rfloor - 1 - \left\lfloor\frac{2n}{3}\right\rfloor \\
                     & = & - 1 \\
                     & \leq & 0
\end{eqnarray*}

That is, less vertices are needed than those that are available, to construct the edge-disjoint paths of the immersion. Therefore, there exists $z \in N_W(u) \cap N_W(v)$ that has not been used in other paths $P_{ux}$ or $P_{vx}$, and then we can assign $P_{uv} = uzv$. Thus, we have obtained a strong immersion of $K_{\lceil\frac{n}{3}\rceil}$ in $G$.

\end{proof}

\subsection{Equivalence of Conjectures \ref{conj2} and \ref{conj3}} \label{SeccionEquivalenciaConjeturas}

The proof of Theorem~\ref{TeoConjeturasEquivalentes} is strongly inspired from \cite{OnASpecialCaseHadwiger}. We will need to use some preliminary results. Suppose Conjecture~\ref{conj2} fails, and let $G$ be a counterexample that minimizes the number of vertices. Observe that the number of vertices is upper bounded by the product of the independence number and the chromatic number. So we have the following inequality: 

\begin{equation} \label{eq2}
|V(G)|\leq2\chi(G).
\end{equation}

We will prove some properties that $G$ satisfies.

\begin{defn}
A graph $G$ is \textit{$k$-color-critical} if $\chi(G)=k$ and $\chi(G-v)<k$, for each $v \in V(G)$.
\end{defn} 

\begin{lema} \label{critic}
$G$ is $\chi(G)$-color-critical.

\begin{proof}
Indeed, if there is a vertex $v \in V(G)$, such that $\chi(G-v)=\chi(G)$, then as $G-v$ has less vertices than $G$ and $\alpha(G-v) \leq 2$, we would have that, $$K_{\chi(G)} = K_{\chi(G-v)} \imm G-v \imm G,$$ which contradicts the fact that $G$ is a counterexample for Conjecture~\ref{conj2}.
\end{proof}
\end{lema}

\begin{lema} \label{compCon}
$\overline{G}$ is connected.

\begin{proof}
If not, $G$ consists of two disjoint subgraphs $G_1$ and $G_2$, such that for all $u \in V(G_1)$ and for all $v \in V(G_2)$, $uv \in E(G)$. Then, as both $G_1$ and $G_2$ have less vertices than $G$, it holds that $K_{\chi(G_1)} \imm G_1$ and $K_{\chi(G_2)} \imm G_2$, and then, $$K_{\chi(G)} = K_{\chi(G_1)+\chi(G_2)} \imm G,$$ which leads to a contradiction. 
\end{proof}
\end{lema}

For the next property, we will use the next result.


\begin{teo}[\cite{CriticalGraphsConnectedComplements}] \label{teoUtil}
If $x$ is any vertex of a $k$-color-critical graph $G$ such that $\overline{G}$ is connected, then $G-x$ has a $(k-1)$-coloring in which every color class contains at least $2$ vertices.
\end{teo}

\begin{lema} \label{numvert}
$|V(G)|=2\chi(G)-1$.  

\begin{proof}
By Lemmas \ref{critic}, \ref{compCon} and Theorem \ref{teoUtil} we know that $G-v$ has a $(\chi(G)-1)$-coloring such that each color class contains at least two vertices. Since $\alpha(G) \leq 2$, each color class in that coloring has size exactly two. So, $|V(G)|=2\chi(G)-1$.  

\end{proof}
\end{lema}

We are now able to prove Theorem~\ref{TeoConjeturasEquivalentes}.

\begin{proof}[Proof of Theorem~\ref{TeoConjeturasEquivalentes}]
By Lemma~\ref{numvert}, we have that $\left \lceil \frac{|V(G)|}{2} \right \rceil = \frac{|V(G)|+1}{2} = \chi(G)$. Then, $K_{\lceil\frac{n}{2}\rceil} \not\imm G$ and therefore, $G$ is also a counterexample for Conjecture~\ref{conj3}. 
\end{proof}

Observe that $G$ turns out to be a counterexample with minimum number of vertices for Conjecture~\ref{conj3} as well. Indeed, let $H$ be a counterexample of Conjecture~\ref{conj3} such that $|V(H)|<|V(G)|$. Then, $$K_{\frac{|V(H)|}{2}} \not\imm H.$$ And as $|V(H)|\leq 2 \chi(H)$, $$K_{\chi(H)} \not\imm H.$$ So, $H$ is a counterexample of Conjecture~\ref{conj2} and $|V(H)|<|V(G)|$, which is a contradiction.

\subsection{Properties of a minimum counterexample of Conjecture~\ref{conj2}} \label{SeccionContraejemploMinimo}

In this subsection we will prove a series of properties that a counterexample of Conjecture~\ref{conj2} with minimum number of vertices satisfies besides those mentioned by Lemmas~\ref{critic},~\ref{compCon} and~\ref{numvert}. The next result enumerates them.



\begin{teo}
Let $G$ be a counterexample to Conjecture~\ref{conj2} which minimizes the number of vertices. Then the following hold:

\begin{enumerate}
\item $G$ is a counterexample to Conjecture~\ref{conj2} which minimizes the chromatic number. \label{MinCromatico}
\item For every $v \in V(G)$, $\overline{G}-v$ has a perfect matching. \label{match}
\item For every pair of nonadjacent vertices $x$, $y$ of $G$, $|N(x) \cap N(y)| \leq \frac{n-1}{2}$. \label{cod}
\item $\omega(G) \geq \frac{n+1}{4}$. \label{MinOmega}
\item $G$ is connected. \label{MinConnected}
\item $\delta(G) \geq \lceil \frac{n}{2} \rceil$. \label{gradmin} 
\item $G$ is Hamiltonian. \label{ham}
\item For every $v \in V(G)$, $G-v$ has a perfect matching. \label{MinMatching}
\item For every $x, y \in V(G)$, it holds that $d(x,y) \leq 2$. \label{dist}
\item $\chi(G) \geq 8$. \label{MinNumeroCromatico8}
\end{enumerate}

Suppose now that $G$, among all counterexamples of Conjecture~\ref{conj2} minimizing the number of vertices, is one that minimizes the number of edges. Then the next additional property hold:

\newcounter{i}
\setcounter{i}{\value{enumi}}
\begin{enumerate}
\setcounter{enumi}{\value{i}}
\item For every edge $e \in E(G)$, it holds that $\alpha(G-e)>\alpha(G)$. \label{LemaAlphaCritic}
\end{enumerate}

\begin{proof} 
\begin{enumerate}
\item[\ref{MinCromatico}.]
Let $\widetilde{G}$ be any counterexample to Conjecture~\ref{conj2} with minimum chromatic number. Then, 
$$2\chi(G)-1 = |V(G)| \leq |V(\widetilde{G})| \leq 2\chi(\widetilde{G}) \leq 2\chi(G).$$ 
Therefore, $\chi(G)=\chi(\widetilde{G})$.

\item[\ref{match}.]
We know by Theorem~\ref{teoUtil} that $G-v$ has a ($\chi(G)-1$)-coloring, in which every color class has exactly two vertices. This corresponds to a perfect matching in $\overline{G}-v$.

\item[\ref{cod}.]
If this is not so, let $x$ and $y$ be two vertices such that $xy \notin E(G)$ and $|N(x) \cap N(y)| \geq \frac{n-1}{2}+1$. As $G$ is a minimal counterexample for Conjecture~\ref{conj3}, we know that $$K_{\frac{n-1}{2}} = K_{\lceil\frac{n-2}{2}\rceil} \imm G-\{x,y\}.$$ 

Let $U$ be the set of corner vertices of such an immersion and let $W=V(G-\{x,y\})\setminus U$. As $\alpha(G)\leq2$ and $xy \notin E(G)$, we have that for every $u \in U$, $ux \in E(G)$ or $uy \in E(G)$. Without loss of generality, assume that $x$ is adjacent to at least half of the vertices in $U$ (and that $x$ has more neighbors than $y$ in $U$). Note that every non-neighbor of $x$ has to be adjacent to $y$. 

Let us see that $x$ is connected to every vertex $u$ in $U$, by edge-disjoint paths $P_{xu}$. If $xu \in E(G)$, then $P_{xu}=xu$. If $xu \notin E(G)$, then $P_{xu}=xzyu$, with $z \in N_W(x)\cap N_W(y)$. Observe that for this to work, it needs to hold that $|N_W(x) \cap N_W(y)| \geq |\overline{N}_U(x)|$. 

We know that $|\overline{N}_U(x)| \leq \frac{n-1}{4}$, so, $$|\overline{N}_U(x)| = \left\lfloor\frac{n-1}{4}\right\rfloor-i \text{ with } i \in \left\{0,...,\left\lfloor\frac{n-1}{4}\right\rfloor\right\}.$$ Besides, $$|N_W(x) \cap N_W(y)| = |N(x) \cap N(y)| - |N_U(x) \cap N_U(y)| \geq |N(x) \cap N(y)|-(2i+1).$$ The last inequality is obtained by assuming that $|N_U(x)| \geq |N_U(y)|$, so the number of neighbors that $x$ and $y$ share in $U$ is bounded. Indeed, $N_U(y) = \overline{N}_U(x) \cup (N_U(x) \cap N_U(y))$ and as we assumed $|N_U(y)| \leq |N_U(x)|$, we have that $$|\overline{N}_U(x)| + |N_U(x) \cap N_U(y)| \leq |N_U(x)|.$$ Then,
\begin{eqnarray*}
|N_U(x) \cap N_U(y)| & \leq & |N_U(x)| - |\overline{N}_U(x)| \\
                     & = & \left(\left\lceil\frac{n-1}{4}\right\rceil + i\right) - \left(\left\lfloor\frac{n-1}{4}\right\rfloor - i\right) \\
                     & \leq & 2i+1.    
\end{eqnarray*}

And as $|N(x) \cap N(y)| \geq \frac{n-1}{2}+1$, we have that $$|\overline{N}_U(x)| = \left\lfloor\frac{n-1}{4}\right\rfloor-i \leq |N(x) \cap N(y)|-(2i+1) \leq |N_W(x) \cap N_W(y)|.$$ 

It is important to notice that the paths $P_{xu}$, from $x$ to $u \in U$, do not interfere with the already existing paths between corner vertices in $U$. This is so, because the new paths only use edges which are incident to $x$ and $y$. Therefore, we get an immersion of $K_{\left\lceil\frac{n-2}{2}\right\rceil+1} = K_{\left\lceil\frac{n}{2}\right\rceil}$ in $G$, which is a contradiction.

\item[\ref{MinOmega}.]
Let $x$, $y$ be any two vertices such that $xy \notin E(G)$ and divide the rest of the vertices into $A = N(x) \setminus N(y)$, $B = N(x) \cap N(y)$ and $C = N(y) \setminus N(x)$. Observe that both $A$ and $C$ induce a complete graph because $\alpha(G)\leq2$. Besides, by property~\ref{cod}, it holds that $|B| \leq \frac{n-1}{2}$. Therefore, at least one of the other two sets, say $A$, satisfies that $|A \cup \{x\}| \geq \frac{n+1}{4}$. And as $\omega(G) \geq |A \cup \{x\}|$, we are done. 

\item[\ref{MinConnected}.]
Indeed, if not so, $G$ would have at least two connected components. In fact, since $\alpha(G) \leq 2$, it would have exactly two connected components and every component would be a complete graph. Then, $K_{\chi(G)} \subseteq G$, which contradicts that $G$ is a counterexample for Conjecture~\ref{conj2}.

\item[\ref{gradmin}.]
Observe first that it is straightforward to prove that $\delta(G) \geq \lfloor\frac{n}{2}\rfloor$, since the non-neighborhood of any vertex induces a complete graph. Indeed, if $\delta(G) < \lfloor\frac{n}{2}\rfloor$, $K_{\lceil \frac{n}{2} \rceil}$ would be a subgraph of $G$, a contradiction.

So suppose that $\delta(G) = \lfloor \frac{n}{2} \rfloor$ and let $v$ be such that $d(v)=\delta(G)$. Divide $V(G)-{v}$ into the neighbors and the non-neighbors of $v$, $N(v)$ and $\overline{N}(v)$ respectively. We know, by property~\ref{match}, that $\overline{G}-v$ has a perfect matching. And given that $\overline{N}(v)$ induces a complete graph, every vertex in $\overline{N}(v)$ is matched with a vertex in $N(v)$. Besides, as $|N(v)|=\lfloor \frac{n}{2} \rfloor$, then $|\overline{N}(v)|=\lfloor \frac{n}{2} \rfloor$. This matching represents a coloring of $G$, in which all color classes have exactly two vertices. We claim that $K_{\chi(G)} \imm G$, where the corner vertices are $\{v\} \cup N(v)$.

Indeed, $vu \in E(G)$, for every $u \in N(v)$. Then, we can assign $P_{vu}=vu$. Consider now $u, w \in N(v)$. If $uw \in E(G)$, then $P_{uw}=uw$. If $uw \notin E(G)$, then, as $\alpha(G) \leq 2$, it holds that $ux_w, wx_u \in E(G)$, where $x_w, x_u$ are the vertices that are matched with $w$ and with $u$, respectively. Also, $x_ux_w \in E(G)$, since $x_u, x_w \in \overline{N}(v)$, which is a complete graph. Therefore, we can assign $P_{uw}=ux_wx_uw$. The paths are edge-disjoint, because by seeing the matching in $\overline{G}-v$ as a coloring in $G-v$, the chosen paths are precisely chains between corner vertices of different colors.

\item[\ref{ham}.]
It follows from property~\ref{gradmin} along with Dirac's Theorem for Hamiltonian graphs \cite{DiracTheo}.

\item[\ref{MinMatching}.]
It is implied by property~\ref{ham} and Lemma~\ref{numvert}.

\item[\ref{dist}.]
There are two cases. If $x,y \in E(G)$, then $d(x,y)=1$. If $x,y \notin E(G)$, the by property~\ref{gradmin}, we know that both $x$ and $y$ have at least $\left\lceil\frac{n}{2}\right\rceil$ neighbors into a set of $n-2$ vertices ($V(G)\setminus\{x,y\}$). That means they have at least one common neighbor, so $d(x,y)=2$.

\item[\ref{MinNumeroCromatico8}.]
It follows directly from Theorem~\ref{ALchico}.

\item[\ref{LemaAlphaCritic}.]
If there were an edge $e \in E(G)$, such that $\alpha(G-e) \leq \alpha(G) = 2$, then, $$K_{\chi(G-e)} \imm G-e \imm G.$$ So, $\chi(G-e) \leq \chi(G)-1$. Therefore, $G-e$ has $|V(G)|=2\chi(G)-1$ vertices and can be colored with $\chi(G)-1$ colors. Necessarily one color class has at least 3 vertices, which is a contradiction.

\end{enumerate}
\end{proof}
\end{teo} 

\section{Conclusion}

The question of whether Abu-Khzam and Langston's conjecture is true still remains open, even in the special case of $\alpha(G) \leq 2$. A possible way would be to continue studying a counterexample of Conjecture~\ref{conj2} minimizing the number of vertices. More structural properties can be found in \cite{Memoria}. 

After seeing the proofs of Theorems~\ref{Teo56Denso} and~\ref{TeoC4Inducido} it is tempting to try to look for an immersion of a complete graph with a vertex of every color as the set of corner vertices and chains as paths between them. However there are examples of graphs with colorings in which it is impossible to find this type of immersion (the reader is referred also to \cite{Memoria}).

\section{Acknowledgments}

The author would like to thank Maya Stein for her tremendous help and support during the creation of this article, and also the two anonymous referees for their very valuable comments.

\bibliographystyle{plain}
\bibliography{bibliografia}
\addcontentsline{toc}{section}{\protect\numberline{}{References}}

\end{document}